\documentclass[11pt]{amsart}
\usepackage[utf8]{inputenc}
\usepackage{amssymb}
\usepackage{amsmath}
\usepackage{amscd}
\usepackage[matrix,arrow]{xy}
\usepackage{url}
\usepackage{color}
\usepackage{xtab}
\usepackage{enumitem}
\usepackage{multirow}

\newcommand{\Z}{{\mathbb Z}}

\begin{document}
\bibliographystyle{plain}
\newtheorem{thm}{Theorem}
\newtheorem{lem}{Lemma}[section]
\newtheorem{prop}[lem]{Proposition}
\newtheorem{alg}[lem]{Algorithm}
\newtheorem{cor}[lem]{Corollary}
\newtheorem*{conj}{Conjecture}

\theoremstyle{definition}

\newtheorem{ex}[thm]{Example}

\theoremstyle{remark}

\newtheorem{ack}{Acknowledgement}

\newtheorem*{rem}{Remark}
\newtheorem*{acknowledgement}{Acknowledgement}
\title[]{Markoff-Rosenberger triples and generalized Lucas sequences}

\author{H.\ R.\ Hashim}
\address{
Institute of Mathematics\\
University of Debrecen\\
P.O.Box 400\\
4002 Debrecen\\
Hungary}

\email{hashim.hayder.raheem@science.unideb.hu}

\author{L. Szalay}
\address{
Department of Mathematics\\
J.~Selye University\\ 
Hradna ul.~21, 94501 Komarno\\ 
Slovakia and
Institute of Mathematics\\
University of Sopron\\
H-9400, Sopron, Bajcsy-Zsilinszky utca 4.\\
Hungary\\
}

\email{szalay.laszlo@uni-sopron.hu}

\author{Sz.\ Tengely}
\address{
Institute of Mathematics\\
University of Debrecen\\
P.O.Box 400\\
4002 Debrecen\\
Hungary}

\email{tengely@science.unideb.hu}

\date{\today}
\thanks{}


\begin{abstract}
We consider the Markoff-Rosenberger equation $$ax^2+by^2+cz^2=dxyz$$ with $(x,y,z)=(U_i,U_j,U_k),$ where $U_i$ denotes the $i$-th generalized Lucas number of first/second kind. We provide upper bound for the minimum of the indices and we apply the result to completely resolve concrete equations, e.g. we determine solutions containing only balancing numbers and Jacobsthal numbers, respectively.
\end{abstract}
\maketitle

\section{Introduction}

Markoff \cite{M1} showed that the equation
$$
x^2+y^2+z^2=3xyz
$$
has infinitely many integral solutions. The equation defined above is called Markoff equation, and it has been generalized in many directions by several authors. In this article, we deal with the generalization 
\begin{equation}\label{MR}
ax^2+by^2+cz^2=dxyz
\end{equation}
considered by Rosenberger \cite{RoseMar}. Rosenberger proved that if $a,b,c,d\in\mathbb{N}$ are integers such that $\gcd(a,b)=\gcd(a,c)=\gcd(b,c)=1$ and $a,b,c|d,$ then non-trivial solutions exist only if $(a,b,c,d)\in A,$ where
$$
A=\{(1,1,1,1),(1,1,1,3),(1,1,2,2),(1,1,2,4),(1,2,3,6),(1,1,5,5)\}.
$$
Luca and Srinivasan \cite{LucaMark} proved that the only solution of Markoff equation with $x\leq y\leq z$ such that $(x,y,z)=(F_i,F_j,F_k)$ is given by the well-known identity related to the Fibonacci numbers
$$
1+F_{2n-1}^2+F_{2n+1}^2=3F_{2n-1}F_{2n+1}.
$$
Kafle, Srinivasan and Togbé \cite{KaSrTo} determined all triples of Pell numbers $(x,y,z)=(P_i,P_j,P_k)$ satisfying the Markoff equation $x^2+y^2+z^2=3xyz.$ Here there is an other identity given by
$$2^2 + P_{2m-1}^2 + P_{2m+1}^2 = 3\cdot 2\cdot P_{2m-1}P_{2m+1}.$$

Markoff-Rosenberger triples containing only Fibonacci numbers were determined by Tengely \cite{TenMR}. Altassan and Luca \cite{AltLuc} considered Markoff-Rosenberger equations with integer solutions $(x, y, z)$ which are all members of a Lucas sequence whose characteristic equation has roots which are quadratic units. In this article, we consider generalized Lucas number solutions of the Markoff-Rosenberger equation. We define the sequences $U_n,V_n$ as follows
\begin{eqnarray*}
U_0(P,Q)=0,U_1(P,Q)=1,&&\quad U_{n}(P,Q)=PU_{n-1}(P,Q)-QU_{n-2}(P,Q),\\
V_0(P,Q)=2,V_1(P,Q)=P,&&\quad V_{n}(P,Q)=PV_{n-1}(P,Q)-QV_{n-2}(P,Q),
\end{eqnarray*}
where neither $P$ nor $Q$ is zero.
\begin{rem}
Assume that $P^\star=-P$, and define
$$U^\star_0=0, \; U^\star_1=1\; U^\star_n=P^\star U^\star_{n-1}-QU^\star_{n-2},$$
$$V^\star_0=2, \; V^\star_1=P^\star, \; V^\star_n=P^\star V^\star_{n-1}-QV^\star_{n-2}.$$
Then we have
$$
U^\star_n=(-1)^{n+1}U_n,\quad V^\star_n=(-1)^nV_n.
$$
Based on the above identities in this paper we only deal with sequences satisfying $P>0.$ 
\end{rem}
In this paper we assume that $0<D=P^2-4Q,P\geq 2$ and $-P-1\leq Q\leq P-1.$ We excluded the cases with $P=1$ to make the presentation simpler. However if $P=1,$ then $-2\leq Q\leq 0.$ Therefore there are only two sequences to be considered. Namely the Fibonacci sequence with $(P,Q)=(1,-1)$ and the Jacobsthal sequence with $(P,Q)=(1,-2).$ The former one was completely solved in \cite{TenMR}, the latter one will be handled separately in this paper. 
The characteristic polynomial associated to the above sequences is given by $x^2-Px+Q.$ The roots of the  characteristic polynomial can be written in the form 
$$
\alpha=\frac{P+\sqrt{D}}{2},\quad \beta=\frac{P-\sqrt{D}}{2}
$$
and we have $\alpha-\beta=\sqrt{D}, \alpha\beta=Q.$ We also assume that $\alpha/\beta$ is not a root of unity. By Binet formulas we have that 
$$
U_n=\frac{\alpha^n-\beta^n}{\alpha-\beta},\quad V_n=\alpha^n+\beta^n.
$$
We assume that there exist constants $s_1,s_2,t_1,t_2,i_1,i_2$ and $j_1,j_2$ such that 
\begin{eqnarray}
s_1\alpha^{k-i_1}\leq & U_k & \leq s_2\alpha^{k+i_2},\label{Us}\\ 
t_1\alpha^{k-j_1}\leq & V_k & \leq t_2\alpha^{k+j_2}\quad\mbox{ for $k\geq 1,$}\label{Vt}
\end{eqnarray}
and these will be fulfilled in the cases that we investigate in this paper.

\section{Main result}
\begin{thm} \label{th111}
Let $(a,b,c,d)\in A$ and 
$$
B_0=\min_{I\in\mathbb{Z}}\left|\alpha^I-\frac{d}{c\sqrt{D}}\right|,\quad B_1=\min_{I\in\mathbb{Z}}\left|\alpha^I-\frac{d}{c}\right|.
$$
If $x=U_i, y=U_j$ and $z=U_k$ with $1\leq i\leq j\leq k$ is a solution of \eqref{MR} and $B_0\neq 0,$ then there exists an effectively computable constant $C_0$ such that $i\leq C_0.$ If $x=V_i, y=V_j$ and $z=V_k$ with $1\leq i\leq j\leq k$ is a solution of \eqref{MR} and $B_1\neq 0,$ then there exists an effectively computable constant $C_1$ such that $i\leq C_1.$ 

\end{thm}
\begin{proof} 
We note that the conditions $P\geq 2, D>0$ and $-P-1\leq Q\leq P-1$ imply that $\alpha>1$ and $|\beta|\leq 1.$ The first inequality is obvious, the second can be proved as follows. Since $P\geq 2$ and $-P-1\leq Q\leq P-1$ we have 
\begin{equation}\label{sopron}
(P-2)^2\leq P^2-4Q\leq (P+2)^2.
\end{equation}
Therefore $P-2\leq \sqrt{D}\leq P+2.$ We have that $\beta=\frac{P-\sqrt{D}}{2},$ hence $$-1\leq \beta\leq 1.$$
We look for solutions satisfying $x=U_i, y=U_j$ and $z=U_k$ with $1\leq i\leq j\leq k.$
We have that 
$$
\frac{c\alpha^k}{\sqrt{D}}-\frac{d}{D}\alpha^{i+j}=-\frac{aU_i^2+bU_j^2}{U_k}+\frac{c\beta^k}{\sqrt{D}}-\frac{d}{D}(\alpha^i\beta^j+\alpha^j\beta^i-\beta^{i+j}).
$$
We apply \eqref{Us} to get an upper bound for $\frac{aU_i^2+bU_j^2}{U_k}$ such that $aU_i^2+bU_j^2\le(a+b)U_j^2$ holds since the Lucas sequence $U_n$ is monotone increasing. We obtain that
$$
\frac{aU_i^2+bU_j^2}{U_k}\leq \frac{(a+b)U_j^2}{U_k}\leq (a+b)\frac{s_2^2}{s_1}\alpha^{2i_2+i_1}\alpha^j.
$$
Since $|\beta|\leq 1,$ we get that 
$$
\left|\frac{c\beta^k}{\sqrt{D}}\right|\leq \left|\frac{c}{\sqrt{D}}\right|\leq \left|\frac{c\alpha^j}{\sqrt{D}}\right|.
$$
The last expression to bound is $\frac{d}{D}(\alpha^i\beta^j+\alpha^j\beta^i-\beta^{i+j}),$ in this case we obtain that 
$$
\left|\frac{d}{D}(\alpha^i\beta^j+\alpha^j\beta^i-\beta^{i+j})\right|\leq \frac{d}{D}(2\alpha^j+1)
$$
hence we have that 
$$
\left|\frac{d}{D}(\alpha^i\beta^j+\alpha^j\beta^i-\beta^{i+j})\right|\leq \frac{3d}{D}\alpha^j.
$$
From the above inequalities we get
$$
\left|\frac{c\alpha^k}{\sqrt{D}}-\frac{d}{D}\alpha^{i+j}\right|\leq \left((a+b)\frac{s_2^2}{s_1}\alpha^{2i_2+i_1}+\frac{c}{\sqrt{D}}+\frac{3d}{D}\right)\alpha^j.
$$
It follows that 
\begin{equation}\label{ijk}
\left|\alpha^{k-i-j}-\frac{d}{c\sqrt{D}}\right|\leq \left((a+b)\frac{s_2^2}{s_1}\alpha^{2i_2+i_1}\frac{\sqrt{D}}{c}+\frac{3d}{c\sqrt{D}}+1\right)\alpha^{-i}.
\end{equation}
Let $$B_0=\min_{I\in\mathbb{Z}}\left|\alpha^I-\frac{d}{c\sqrt{D}}\right|.$$
If $B_0\neq 0,$ then we get an upper bound for $i$ from the inequality
\begin{equation}\label{boundi}
\alpha^i\leq \frac{1}{B_0}\left((a+b)\frac{s_2^2}{s_1}\alpha^{2i_2+i_1}\frac{\sqrt{D}}{c}+\frac{3d}{c\sqrt{D}}+1\right).
\end{equation}
In a similar way one can prove the second part of the statement, here we note that we get the inequalities (assuming $B_1\neq 0$)
\begin{eqnarray}
\left|\alpha^{k-i-j}-\frac{d}{c}\right|&\leq& \left((a+b)\frac{t_2^2}{t_1}\alpha^{2j_2+j_1}\frac{1}{c}+\frac{3d}{c}+1\right)\alpha^{-i}.\label{Vijk}\\ 
\alpha^i&\leq& \frac{1}{B_1}\left((a+b)\frac{t_2^2}{t_1}\alpha^{2j_2+j_1}\frac{1}{c}+\frac{3d}{c}+1\right).\label{Vboundi}
\end{eqnarray}
\end{proof}
We can also classify the cases satisfying $B_0=0$, or $B_1=0$, the results are as follows.
\begin{thm}\label{B0}
If $P\ge2$, then $B_0\ne0$ fulfils unless
\begin{itemize}
	\item $P=3,Q=2$, $\alpha=2$, $\sqrt{D}=1$, $e=1$, $I=0$,
	\item $P=3,Q=2$, $\alpha=2$, $\sqrt{D}=1$, $e=2$, $I=1$,
	\item $P=4,Q=3$, $\alpha=3$, $\sqrt{D}=2$, $e=2$, $I=0$,
	\item $P=5,Q=4$, $\alpha=4$, $\sqrt{D}=3$, $e=3$, $I=0$,
\end{itemize}
where $e=d/c$.
\end{thm} 
\begin{proof}
Suppose that $B_0=0$. We distinguish two cases.

Case 1: $\sqrt{D}$ is irrational.
\smallskip

Since $\alpha^I=d/(c\sqrt{D})$ holds for some $I\in\Z$, and $\beta$ is the conjugate of $\alpha$ we have $\beta^I=-d/(c\sqrt{D})$. Subsequently,
$$
V_I=\alpha^I+\beta^I=0.
$$
This is impossible because $V_0=2$, $V_1=P$, and the sequence $V_n$ is monotone increasing (we assumed $P\geq 2$).

Case 2: $\sqrt{D}$ is rational.
\smallskip

Assume that $\sqrt{D}=\sqrt{P^2-4Q}=\delta\in\Z^+$  (i.e.~$\delta\ge 1$.) Now $\alpha=(P+\delta)/2$, and we need to consider the equation
\begin{equation}\label{c2A}
\left(\frac{P+\delta}{2}\right)^I=\frac{e}{\delta},
\end{equation}
where $e=d/c\in\{1,2,3\}$ according to set $A$ appears on the first page of the paper. Now equation (\ref{c2A}) is considered for different values of $I$.

Suppose first that $I\ge1$. Since $P\ge2$, then
$$
\left(\frac{P+\delta}{2}\right)^I\ge\left(\frac{3}{2}\right)^I>\frac{e}{\delta} 
$$
follows if either $I\ge3$, or $I\ge2$ and $e\ne3$, or $I\ge1$ and $e\ne2,3$. The three remaining possibilities can be easily handled by considering the prime decomposition of the right hand side of the equivalent form 
$$
\delta\cdot(P+\delta)^I=e\cdot 2^I
$$
of (\ref{c2A}). Clearly, on the left hand side $\delta<P+\delta$, moreover $Q\ne0$, and $-P-1\leq Q\leq P-1$fulfil. Under these conditions only one solution exist to (\ref{c2A}), namely the second item of the theorem.

Now turn to the case $I=0$. It implies $\delta=e$. Then condition (\ref{sopron}) provides finitely many values for $P$. Thus one can check the existence of the nonzero integer $Q$ satisfying $P^2-4Q=\delta^2$. This verification gives the items 1, 3 and 4 of the theorem.

Assume $I=-1$. From (\ref{c2A}) we conclude $(2-e)\delta=eP$, and it leads to a contradiction with $e\in\{1,2,3\}$.

Finally, suppose $I\le-2$. Put $J=-I$. Obviously, $2\le J$ satisfies
$$
\delta\cdot2^J=e\cdot(P+\delta)^J
$$ 
for some $\delta$, $e$ and $P$. The left hand side is smaller than the right hand side if $\delta=1$. Hence $\delta\ge2$. In addition, $e\mid \delta$ ($e$ is prime or 1). Consequently, there exists a positive integer $\delta_1\ge2$ such that
$\delta=e\delta_1$. Thus we have
$\delta_1\cdot2^J=(P+e\delta_1)^J$.
It follows now that $\delta_1=\gamma^J$ for some positive integer $\gamma\ge2$. Then the contradiction is immediately implied by
$$
2\gamma=P+e\gamma^J.
$$
\end{proof}

\begin{thm}\label{B1}
Put again $e=d/c$. If $P\ge2$, then $B_1\ne0$ fulfils unless
\begin{itemize}
	\item $P\ge2,-P-1\le Q\le P-1$, $e=1$, $I=0$,
	\item $P=3,Q=2$, $\alpha=2$, $\sqrt{D}=1$, $e=2$, $I=1$,
	\item $P=4,Q=3$, $\alpha=3$, $\sqrt{D}=2$, $e=3$, $I=1$,
	\item $P=2,Q=-3$, $\alpha=3$, $\sqrt{D}=4$, $e=3$, $I=1$.
\end{itemize}
\end{thm} 
\begin{proof}
The proof is similar, but simpler, to the proof of the previous theorem. Obviously $I=0$ is always suitable independently of $\alpha$ if $e=1$. Otherwise, if $I\ne0$, then $\alpha$ must be integer. It is sufficient to verify the equation $((P+\delta)/2)^I=1,2,3$ with the known conditions.
\end{proof}

\section{Applications}
For a given tuple $(a,b,c,d)$ and a sequence $R_n$, where $R_n$ is either $U_n$ or $V_n$, Theorem \ref{th111} provides an upper bound for $i,$ denote it by $\mathfrak{ub}_{R_n}(a,b,c,d).$ The strategy described in \cite{TenMR} in case of the Fibonacci sequence can be applied to determine the solutions of equation \eqref{MR} with $x=R_i,y=R_j,z=R_k.$ We note that in the theorem $1\leq i\leq j\leq k$ is assumed, hence to resolve the equation with e.g. $(a,b,c,d)=(1,1,2,2)$ one needs to handle the cases with $i\leq \mathfrak{ub}_{R_n}(1,1,2,2), i\leq \mathfrak{ub}_{R_n}(1,2,1,2)$ and $i\leq \mathfrak{ub}_{R_n}(2,1,1,2).$ Then after obtaining the solutions of \eqref{MR} with these cases we permute the components of these solutions in which they satisfy equation \eqref{MR} at the tuple $(a,b,c,d)=(1,1,2,2)$ in order to determine all of its solutions $(x,y,z)=(R_i, R_j, R_k).$

If we fix $(a,b,c,d)$, $i$ and $m=k-j,$ then  we need to study the equation
$$
aR_i^2+bR_j^2+c R_{j+m}^2-dR_i R_j R_{j+m}=0,
$$
where $R_n=U_n$ or $V_n.$
We note that the equation above only depends on $j.$ 
We adopt the arguments given in \cite{TenMR}.
\begin{enumerate}[label=(\Roman*)]
\item \label{MRargu1} We eliminate as many values of $i$ as possible by checking solvability of quadratic equations 
$$
aR_i^2+by^2+cz^2-dR_iyz=0.
$$
\item \label{MRargu2}  For fixed $m$ we eliminate equations $aR_i^2+bR_j^2+c R_{j+m}^2-dR_i R_j R_{j+m}=0$ modulo $p,$ where $p$ is a prime.
\item  \label{MRargu3} We can also eliminate equations $aR_i^2+bR_j^2+c R_{j+m}^2-dR_i R_j R_{j+m}=0$ using related identities of second order linear
recurrence sequences.
\item \label{MRargu4} We consider the equation $aR_i^2 + bR_j^2 + cR_{j+m}^2 = dR_i R_j R_{j+m} $ as a quadratic in $R_j.$ Its discriminant 
$d^2R_i^2R_{j+m}^2 -4b(aR_i^2 + cR_{j+m}^2 )$ must be a square. 
The sequences $U_n$ and $V_n$ satisfy  the fundamental identity 
$$
V_n^2-DU_n^2=4Q^n.
$$
Therefore in case of $Q=\pm 1$ we have the systems of equations
\begin{eqnarray*}
Y_1^2&=&DX^2\pm 4,\\
Y_2^2&=&d^2R_i^2X^2 -4b(aR_i^2 + cX^2 ),
\end{eqnarray*}
where $X=R_{j+m}=U_{j+m},$
and 
\begin{eqnarray*}
Y_1^2&=&DX^2\mp 4D,\\
Y_2^2&=&d^2R_i^2X^2 -4b(aR_i^2 + cX^2 ),
\end{eqnarray*}
where $X=R_{j+m}=V_{j+m}.$ Multiplying these equations together, in general, yields quartic genus 1 curves. One may determine the integral points on these curves using the Magma \cite{MAGMA} function  (based on results obtained by Tzanakis \cite{Tz1996}) \texttt{SIntegralLjunggrenPoints}. In specific cases it is only quadratic, if, in the first case $d^2U_i^2-4bc=0$, i.e. $d=1, U_i=2  (P=2)$. Similarly, for the sequence $V_i=2$ is possible.
\end{enumerate}

Let us apply the results of Theorem \ref{th111} with these arguments to determine the solutions of equation \eqref{MR} in some second order linear recurrence sequences.
\subsection{Balancing numbers and Markoff-Rosenberger equations}

The first definition of balancing numbers is essentially due to Finkelstein \cite{MR190079}, although he called them numerical centers. In 1999 Behera and Panda  \cite{BePa} defined balancing numbers as follows. A positive integer $n$ is called a balancing number if $$1 + 2 +\ldots+ (n-1) = (n + 1) + (n + 2) +\ldots+ (n + k)$$ for some $k\in\mathbb{N}.$ The sequence of balancing numbers is denoted by $B_n$ for $n\in\mathbb{N}.$ This sequence can be defined in a recursive way as well, we have that $B_0=0,B_1=1$ and 
$$
B_{n}=6B_{n-1}-B_{n-2},\quad n\geq 2.
$$
As we see this is the sequence $U_n(6,1).$ So $P=6, Q=1$ and $D=32.$ We also have that 
$$
\alpha=3+2\sqrt{2},\quad \beta=3-2\sqrt{2}.
$$
We have the bounds 
$$
\alpha^{n-1}\leq B_n\leq \alpha^n,
$$
that is $i_1=1,i_2=0,s_1=s_2=1.$
Since $Q=1$ the numbers $X=B_n$ satisfy the Pellian equation $Y^2=8X^2+1.$ We prove the following result.
\begin{thm}
If $(x,y,z)=(B_i,B_j,B_k)$  is a solution of the equation  
$$
ax^2+by^2+cz^2=dxyz
$$
and $(a,b,c,d)\in\{(1,1,1,1),(1,1,1,3),(1,1,2,2), (1,1,2,4), (1,1,5,5), (1,2,3,6)\},$ then there is at most one solution given by 
$
x=y=z=B_1=1.
$
\end{thm}
\begin{proof}
In the proof we compute the bounds $C_0$ for the different tuples $(a,b,c,d).$ It turns out that in all the cases we have $i\leq 6.$ The argument given by (I) is very useful here, many values of $i$ can be eliminated by checking integral solutions of binary quadratic forms. Therefore we skip the congruence arguments given by (II) and (III). We directly consider the genus 1 curves obtained from the system of equations
\begin{eqnarray*}
Y_1^2&=&8X^2+1,\\
Y_2^2&=&d^2R_i^2X^2 -4b(aR_i^2 + cX^2 ).
\end{eqnarray*}
In the following table we provide details of the computations.
\small{
\begin{center}
\begin{tabular}{|c|c|c|c|c|}
\hline
$\left[a,b,c,d\right]$ & $B_0$& $C_0$ & $\left[i\right]$ & $\left[i, A'X^{4} +B'X^{2} +C', \left[X,Y\right]\right]$ \\
\hline
$\left[1, 1, 1, 1\right]$ & $0.086511$ & $5$ & $\left[2\right]$ & $\left[2, 256X^{4} - 1120X^{2} - 144, \left[\right]\right]$ \\
\hline
$\left[1, 1, 1, 3\right]$ & $0.22989$ & $4$ & $\left[1\right]$ & $\left[1, 40X^{4} - 27X^{2} - 4, \left[\left[1, -3\right], \left[-1, -3\right]\right]\right]$ \\
\hline
$\left[1, 1, 2, 2\right]$ & $0.086511$ & $4$ & $\left[2\right]$ & $\left[2, 1088X^{4} - 1016X^{2} - 144, \left[\right]\right]$ \\
\hline
$\left[1, 2, 1, 2\right]$ & $0.053111$ & $5$ & $\left[2\right]$ & $\left[2, 1088X^{4} - 2168X^{2} - 288, \left[\right]\right]$ \\
\hline
$\left[2, 1, 1, 2\right]$ & $0.053111$ & $5$ & $\left[\right]$ & $\left[\right]$ \\
\hline
$\left[1, 1, 2, 4\right]$ & $0.053111$ & $4$ & $\left[1\right]$ & $\left[1, 64X^{4} - 24X^{2} - 4, \left[\left[1, 6\right], \left[-1, 6\right]\right]\right]$ \\
\hline
$\left[1, 2, 1, 4\right]$ & $0.29289$ & $4$ & $\left[1\right]$ & $\left[1, 64X^{4} - 56X^{2} - 8, \left[\left[1, 0\right], \left[-1, 0\right]\right]\right]$ \\
\hline
$\left[2, 1, 1, 4\right]$ & $0.29289$ & $4$ & $\left[1\right]$ & $\left[1, 96X^{4} - 52X^{2} - 8, \left[\left[1, -6\right], \left[-1, -6\right]\right]\right]$ \\
\hline
$\left[1, 1, 5, 5\right]$ & $0.086511$ & $3$ & $\left[1\right]$ & $\left[1, 40X^{4} - 27X^{2} - 4, \left[\left[1, -3\right], \left[-1, -3\right]\right]\right]$ \\
\hline
$\left[1, 5, 1, 5\right]$ & $0.11612$ & $5$ & $\left[1\right]$ & $\left[1, 40X^{4} - 155X^{2} - 20, \left[\left[2, 0\right], \left[-2, 0\right]\right]\right]$ \\
\hline
$\left[5, 1, 1, 5\right]$ & $0.11612$ & $5$ & $\left[1\right]$ & $\left[1, 168X^{4} - 139X^{2} - 20, \left[\left[1, -3\right], \left[-1, -3\right]\right]\right]$ \\
\hline

\multirow{ 2}{*}{$\left[1, 2, 3, 6\right]$} & \multirow{ 2}{*}{$0.053111$} & \multirow{ 2}{*}{$4$} & \multirow{ 2}{*}{$\left[1, 3\right]$} & $\left[1, 96X^{4} - 52X^{2} - 8, \left[\left[1, -6\right], \left[-1, -6\right]\right]\right]$\\  
& & & & $\left[3, 352608X^{4} - 34324X^{2} - 9800, \left[\right]\right]$ \\
\hline

\multirow{ 2}{*}{$\left[1, 3, 2, 6\right]$} & \multirow{ 2}{*}{$0.22989$} & \multirow{ 2}{*}{$4$} & \multirow{ 2}{*}{$\left[1, 3\right]$} & $\left[1, 96X^{4} - 84X^{2} - 12, \left[\left[1, 0\right], \left[-1, 0\right]\right]\right]$\\ 
 & & & & $\left[3, 352608X^{4} - 73524X^{2} - 14700, \left[\right]\right]$ \\
\hline
$\left[2, 1, 3, 6\right]$ & $0.053111$ & $4$ & $\left[1\right]$ & $\left[1, 192X^{4} - 40X^{2} - 8, \left[\left[1, 12\right], \left[-1, 12\right]\right]\right]$ \\
\hline
$\left[2, 3, 1, 6\right]$ & $0.060659$ & $6$ & $\left[1\right]$ & $\left[1, 192X^{4} - 168X^{2} - 24, \left[\left[1, 0\right], \left[-1, 0\right]\right]\right]$ \\
\hline
$\left[3, 1, 2, 6\right]$ & $0.22989$ & $4$ & $\left[1\right]$ & $\left[1, 224X^{4} - 68X^{2} - 12, \left[\left[1, -12\right], \left[-1, -12\right]\right]\right]$ \\
\hline
$\left[3, 2, 1, 6\right]$ & $0.060659$ & $6$ & $\left[1\right]$ & $\left[1, 224X^{4} - 164X^{2} - 24, \left[\left[1, 6\right], \left[-1, 6\right]\right]\right]$ \\
\hline
\end{tabular}
\end{center}
}
The first column gives the tuple $(a,b,c,d),$ in the second column we see the bounds $B_0,$ in the third column we have the bounds $C_0$ for $i,$ in the fourth column we provide the lists containing the remaining values of $i$ not eliminated by the argument (I), in the last column we have the list containing $i,$ the right hand side of the quartic polynomial $Y^2=A'X^{4} +B'X^{2} +C'$ defining the genus 1 curve, and the integral solutions (the second coordinate is only up to sign, for us, only the first coordinate is interesting, since that gives $B_{j+m}$). For example in case of $(a,b,c,d)=(1,2,3,6)$ we have that the bound $B_0=0.053111$ and $C_0=4.$ That is $i\leq 4.$ Applying argument (I) we can eliminate $i=2,4,$ hence the list of remaining values is given by $[1,3].$ If $i=1,$ then we get that the only integral solutions are the ones with $X=\pm 1.$ Since $X=B_{j+m}$ the only possibility is $B_{j+m}=1.$ The last step is the solution of the quadratic equation
$$
1^2+2\cdot B_j^2+3\cdot 1^2=6\cdot 1\cdot B_j\cdot 1.
$$
It follows that $B_j$ is either 1 or 2, but 2 is not a balancing number. Therefore the only solution in this case 
$$
(B_i,B_j,B_k)=(B_1,B_1,B_1)=(1,1,1).
$$
\end{proof}

\subsection{Jacobsthal numbers and Markoff-Rosenberger equations} \label{Jacobexam}
If $(P,Q)=(1,-2),$ then we deal with the so-called Jacobsthal numbers $J_n=U_n(1,-2).$ That is we have $J_0=0,J_1=1$ and
$$
J_n=J_{n-1}+2J_{n-2}\quad\mbox{ if }n\geq 2.
$$
We obtain that $$D=9, \quad \alpha=2, \quad \beta=-1.$$ Therefore the closed-form of $J_n$ is given by
$$
\frac{2^n-(-1)^n}{3}.
$$
Based on the above closed-form equation we may provide bounds for $J_n,$ these are as follows
$$
\frac{2^{n-1}}{3}\leq J_n\leq 2^{n-1}, \quad n\geq 1.
$$
We get that $i_1=1, i_2=-1, s_1=1/3$ and $s_2=1.$ We prove the following statement.
\begin{thm}
If $(x,y,z)=(J_i,J_j,J_k)$  is a solution of equation  
\begin{equation}\label{JacMR}
aJ_i^2+bJ_j^2+cJ_k^2=dJ_iJ_jJ_k
\end{equation}
and $(a,b,c,d)\in\{(1,1,1,1), (1,1,1,3), (1,1,2,2), (1,1,2,4), (1,1,5,5), (1,2,3,6)\},$ then the complete list of solutions are given by
\begin{center}
\begin{tabular}{|c|c|}
	\hline 
	$(a,b,c,d)$	& solutions \\ 
	\hline 
	$(1,1,1,1)$	& $\{(3,3,3)\}$ \\ 
	\hline 
	$(1,1,1,3)$ & $\{(1,1,1)\}$ \\
	\hline
	$(1,1,2,2)$	& $\{ \}$ \\ 
	\hline 
	$(1,1,2,4)$	& $\{(1,1,1),(1,3,1),(1,3,5),(3,1,1),(3,1,5),(3,11,1),(11,3,1)\}$ \\
	\hline  
   $(1,1,5,5)$	& $\{(1,3,1),(3,1,1)\}$ \\ 
	\hline 	
	$(1,2,3,6)$	& $\{(1,1,1),(5,1,1)\}$ \\ 
	\hline 
\end{tabular}. 	
\end{center}
\end{thm}
\begin{proof}
Since $\beta=-1,$ we may follow the steps of the proof of Theorem 1. We have that $\sqrt{D}=3,$ so in some cases we obtain $B_0=0.$

{\bf The case $(a,b,c,d)=(1,1,1,1).$} Here we obtain that $B_0\approx 1.666$ and the bound for $i$ is 5. Applying the argument given at (I) it turns out that all values can be eliminated except $i=3.$ If $i=3,$ then we compute the possible values of $k-j$ from inequality \eqref{ijk}. We have that $k-j\in\{0,1,2,3\}.$ If $k-j\in\{1,2\},$ then applying (II) with $p=3$ works and in case of $k-j=3$ we use $p=11$ to show that there is no solution. The remaining case is related to $k-j=0.$ We obtain the equation
$$
3^2+J_j^2+J_j^2=3J_jJ_j.
$$
It follows that $J_j=J_k=3,$ so the solution is given by $(J_i,J_j,J_k)=(3,3,3).$

{\bf The case $(a,b,c,d)=(1,1,1,3).$} In this case in \eqref{ijk} we have $|2^{k-i-j}-1|$ and this expression is 0 if $k-i-j=0.$ Therefore we need to study the equation
\begin{multline*}
(2^i-(-1)^i)^2+(2^j-(-1)^j)^2+(2^{i+j}-(-1)^{i+j})^2=\\(2^i-(-1)^i)\cdot (2^j-(-1)^j)\cdot (2^{i+j}-(-1)^{i+j}).
\end{multline*}
By symmetry we may assume that $i\leq j.$ The small solutions with $0\leq i\leq j\leq 2$ can be enumerated easily. Since we consider solutions with $i,j>0$ we omit $(i,j)=(0,0).$ The other solution is given be $(i,j)=(1,1),$ hence we get that $(J_i,J_j,J_k)=(1,1,1).$ If $i=2,$ then it follows modulo 7 that there is no solution. If $i>2,$ then we work modulo 8 to show that no solution exists. If $k-i-j\neq 0,$ then we obtain that 
 $$B_0=\min_{I\in\mathbb{Z}\setminus\{0\}}\left|\alpha^I-\frac{d}{c\sqrt{D}}\right|=1.$$
As a consequence we have that $i\in\{1,2\}.$ We may exclude the cases $i=1,2$ and $k-j=2$ modulo 5. In a similar way working modulo 7 we eliminate the cases $i=2,k-j=3$ and $i=3,k-j=3,4.$ The remaining cases are given by $i\in\{1,2\},k-j\in\{0,1\}.$
If $i=1,2, k-j=0,$ then it easily follows that $(1,1,1)$ is the only solution. If $i=1,2,k-j=1,$ then the equation is 
$$
1+J_j^2+J_{j+1}^2=3J_jJ_{j+1}.
$$
Since $J_{j+1}=2J_j+(-1)^j,$ the above equation can be written as
$$
J_j^2-(-1)^jJ_j-2=0.
$$
Thus the only possibilities are given by $J_j\in\{\pm 1,\pm 2\}.$ Again the only solution we get is $(1,1,1).$

{\bf The case $(a,b,c,d)=(1,1,2,2).$} Here we compute the bounds for $i$ in the cases $(a,b,c,d)=(1,1,2,2),(1,2,1,2),(2,1,1,2).$ Simply argument (I) is enough to show that there exists no solution.

{\bf The case $(a,b,c,d)=(1,1,2,4).$} The bound for $i$ is 4 and by (I) we can eliminate the case $i=4$ when the order of the coefficients is $(1,1,2,4).$ Congruence arguments (modulo 3 or 7) work if $(i,k-j)\in\{(1,2),(1,3),(2,2),(2,3)\}.$ The remaining cases are
$$
(i,k-j)\in\{(1,0),(1,1),(2,0),(2,1),(3,0),(3,1),(3,2),(3,3)\}.
$$ 
From $(i,k-j)=(1,0),(2,0),(3,0)$ we obtain the solutions (by solving quadratic equations) $(1,1,1)$ and $(3,1,1).$ If $(i,k-j)=(1,1),(2,1),$ then we get $J_j^2+4(-1)^jJ_j+3=0.$ Hence $J_j=1$ or 3. So we obtain the solutions $(1,1,1),(1,3,1),(1,3,5).$
In case of $(i,k-j)=(3,1)$ we obtain $15J_j^2+4(-1)^jJ_j-11=0.$ Thus we have the solution $(3,1,1).$ By applying the rule $J_{n+1}=2J_{n}+(-1)^n$ two or three times we can reduce the problems $(i,k-j)=(3,2),(3,3)$ to quadratic equations. The formulas are getting more involved, for example if $(i,k-j)=(3,2)$ we have
$$
9+J_j^2+2(4J_j+2(-1)^j+(-1)^{j+1})^2=12J_j(4J_j+2(-1)^j+(-1)^{j+1}).
$$
In this case we get that $J_j=1.$ In a very similar way we handle the cases with the tuples $(1,2,1,4)$ and $(2,1,1,4).$

{\bf The case $(a,b,c,d)=(1,1,5,5).$} Here we need to deal with the tuples $(1,1,5,5),(1,5,1,5)$ and $(5,1,1,5).$ The bounds for $i$ are given by 3, 6 and 6, respectively. Since the steps are similar as we have applied in the previous cases, we omit the details.

 {\bf The case $(a,b,c,d)=(1,2,3,6).$} We only provide some data related to the computation. Let us start with the bounds:
 \begin{center}
 \begin{tabular}{|c|c|c|}
 \hline 
 tuple & bound for $i$ & special case \\ 
 \hline 
 $(1,2,3,6)$ &4 & - \\ 
 \hline 
 $(1,3,2,6)$ & 2 & $k-i-j=0$ \\ 
 \hline 
 $(2,1,3,6)$ & 4 & - \\ 
 \hline 
 $(2,3,1,6)$ & 2 & $k-i-j=1$ \\ 
 \hline 
 $(3,1,2,6)$ & 2 & $k-i-j=0$ \\ 
 \hline 
 $(3,2,1,6)$ & 2 & $k-i-j=1$ \\ 
 \hline 
 \end{tabular}.
 \end{center}
As before we apply the arguments given by (I) and (II) and the identity $J_{n+1}=2J_n+(-1)^n$ to resolve all the possible cases. The only new case that has not appeared yet is $k-i-j=1.$ If we take the tuple $(2,3,1,6),$ then we obtain
\begin{equation}\label{eq1}
2J_i^2+3J_j^2+J_{i+j+1}^2-6J_iJ_jJ_{i+j+1}=0,
\end{equation}
or
\begin{multline}\label{maincase}
2(2^i-(-1)^i)^2+3(2^j-(-1)^j)^2+(2^{i+j+1}-(-1)^{i+j+1})^2=\\2(2^i-(-1)^i)\cdot (2^j-(-1)^j)\cdot (2^{i+j+1}-(-1)^{i+j+1}).
\end{multline}

With respect to the values of $i$ and $j$ we consider the following cases (with assuming that $1 \leq i \leq j \leq k$):
\begin{itemize}
	\item If  $i$ and $j$ are both even, i.e. $i=2t$ and $j=2r$ for all positive integers $t, r \geq 1$ then equation \eqref{maincase} becomes 
\begin{multline*}
E_1=2(4^t-1)^2+3(4^r-1)^2+(2\cdot 4^{t+r}+1)^2-\\2(4^t-1)\cdot(4^r-1)\cdot(2\cdot4^{t+r}+1)=0.
\end{multline*}	
But, $E_1 \pmod{8} \equiv 4$ for all  $t, r  \geq 1$, which leads to a contradiction. Moreover, since $i$ and $j$ are both even with $i \geq 2$ and $j \geq 2$ then all the even values of $i$ and $j$ are excluded.
	\item If  $i$ and $j$ are both odd, i.e. $i=2t+1$ and $j=2r+1$ for all positive integers $t, r \geq 1$ then equation \eqref{maincase} implies to
\begin{multline*}
E_2=2(2\cdot4^t+1)^2+3(2\cdot4^r+1)^2+(2\cdot4^{t+r+1}+1)^2-\\2(2\cdot4^t+1)(2\cdot4^r+1)(2\cdot4^{t+r+1}+1)=0.
\end{multline*}
Similarly, $E_2 \pmod{8} \equiv 4$ for all  $t, r  \geq 1$, and again we get a contradiction. Indeed, all the odd values of $i \geq 3$ and $j \geq 3$ are excluded, and it remains only to check whether equation \eqref{eq1} has solutions or not at the following cases: $i=1, j=1; i=1, j\geq 3; j=1, i\geq 3.$ In fact, since we assumed that $1 \leq i \leq j \leq k$ then the later case can be covered by checking the solvability of equation \eqref{eq1} at $i=j=1.$
\item If  $i$ is even and $j$ is odd, i.e. $i=2t$ and $j=2r+1$ for all positive integers $t, r \geq 1$ then equation \eqref{maincase} leads to 
\begin{multline*}
E_3=2(4^t-1)^2+3(2\cdot4^r+1)^2+(4^{t+r+1}-1)^2-\\ 2(4^t-1)(2\cdot4^r+1)(4^{t+r+1}-1)=0.
\end{multline*}
Again, we get a contradiction since $E_3 \pmod{8} \equiv 4$ for all $t, r  \geq 1$. Here, we excluded all the even values of $i \geq 2$ and odd values of $j \geq 3$, and it remains to check whether equation \eqref{eq1} has solutions or not only at $j=1, i\geq 2$. Similarly, this can be covered by studying the solutions of equation \eqref{eq1} only at $i=j=1.$ 

\item Finally, if $i$ is odd and $j$ is even, i.e. $i=2t+1$ and $j=2r$ for all positive integers $t, r  \geq 1$ then similarly we have 
\begin{multline*}
E_4=2(2\cdot4^t+1)^2+3(4^r-1)^2+(4^{t+r+1}-1)^2-\\ 2(2\cdot4^t+1)\cdot(4^r-1)\cdot(4^{t+r+1}-1)=0,
\end{multline*}
and $E_4 \pmod{8} \equiv 4$ for all $t, r \geq 1$, which gives a contradiction. It is clear that all the odd values of $i \geq 3$ and even values of $j \geq 2$  are excluded, and we need to check whether equation \eqref{eq1} has solutions or not only at $i=1, j\geq 2$. 
\end{itemize}
From these cases we conclude that it only remains to study the solutions of equation \eqref{eq1} at $i=1$ and all the integers of $j$ with $j\geq 1.$ This can be done by direct substitution and using argument (III) as follows.\\
 
It is clear that we have $k=i+j+1=j+2$ and 
\begin{equation}\label{111}
2+3J_j^2+J_{j+2}^2-6J_jJ_{j+2}=0 \quad \mbox{for} \quad  j \geq 1.
\end{equation}
\begin{itemize}
\item If $j=1$ then we have that $-4=2+3J_1^2+J_{3}^2-6J_1J_{3}=0$, which is impossible.
\item If $j=2$ then  we get the solution $(i,j,k)=(1,2,4)$. Hence, equation \eqref{eq1} has the solution $(J_i, J_j, J_{i+j+1})=(J_i, J_j, J_{k})=(J_1, J_2, J_4)=(1,1,5).$
\item If $j \geq 3$, we can show that equation \eqref{eq1} has no more solutions by showing that

$$
2+3J_j^2+J_{j+2}^2-6J_jJ_{j+2}<0 \quad \mbox{for} \quad j \geq 3.
$$
Indeed, after substituting the Jacobsthal numbers formula in the left hand side of equation \eqref{111} a few times we get that
\begin{equation*}
\begin{split}
2+3J_j^2+J_{j+2}^2-6J_jJ_{j+2} &=2-2J_{j-1}^2-24J_{j-1}J_{j-2}-24J_{j-2}^2\\
&<0 \quad \mbox{for}\quad  j \geq 3,
\end{split}
\end{equation*}
and this contradicts equation \eqref{111}.
\end{itemize}
Therefore, by permuting the components of the solution $(1,1,5)$ to be a solution of equation \eqref{JacMR} at the tuple $(1,2,3,6)$ we get the solution $(5,1,1)$.
\end{proof}

\begin{acknowledgement}
The research was supported in part by grants ANN130909, K115479 and K128088 of the Hungarian National Foundation for Scientific Research (Sz.~T.).  For L.~Sz.~the research was supported by Hungarian National Foundation for Scientific Research Grant No.~128088. This presentation has been made also in the frame of the ``{\sc Efop}-3.6.1-16-2016-00018 -- Improving the role of the research $+$ development $+$ innovation in the higher education through institutional developments assisting intelligent specialization in Sopron and Szombathely''. The work of H. R. Hashim was supported by the Stipendium Hungaricum Scholarship.
\end{acknowledgement}
\bibliography{allbibMR}

\newcommand{\noop}[1]{} \def\cprime{$'$}
\begin{thebibliography}{10}

\bibitem{AltLuc}
A.~Altassan and F.~Luca.
\newblock Markov type equations with solutions in {L}ucas sequences.
\newblock \noop{3001}submitted.

\bibitem{BePa}
A.~Behera and G.~K. Panda.
\newblock On the square roots of triangular numbers.
\newblock {\em Fibonacci Quart.}, 37(2):98--105, 1999.

\bibitem{MAGMA}
W.~Bosma, J.~Cannon, and C.~Playoust.
\newblock The {M}agma algebra system. {I}. {T}he user language.
\newblock {\em J. Symbolic Comput.}, 24(3-4):235--265, 1997.
\newblock Computational algebra and number theory (London, 1993).

\bibitem{MR190079}
R.~Finkelstein.
\newblock The house problem.
\newblock {\em Amer. Math. Monthly}, 72:1082--1088, 1965.

\bibitem{KaSrTo}
B.~Kafle, A.~Srinivasan, and A.~Togb\'e.
\newblock Markoff equation with {P}ell components.
\newblock {\em The Fibonacci Quarterly}, \noop{3001}to appear.

\bibitem{LucaMark}
F.~Luca and A.~Srinivasan.
\newblock Markov equation with {F}ibonacci components.
\newblock {\em Fibonacci Quart.}, 56(2):126--129, 2018.

\bibitem{M1}
A.~Markoff.
\newblock Sur les formes quadratiques binaires ind\'{e}finies.
\newblock {\em Math. Ann.}, 17(3):379--399, 1880.

\bibitem{RoseMar}
G.~Rosenberger.
\newblock \"{U}ber die diophantische {G}leichung {$ax^{2}+by^{2}+cz^{2}=dxyz$}.
\newblock {\em J. Reine Angew. Math.}, 305:122--125, 1979.

\bibitem{TenMR}
Sz. Tengely.
\newblock {M}arkoff-{R}osenberger triples with {F}ibonacci components.
\newblock {\em Glas. Mat. Ser. III}, 55(1):265--273, 2020.
\newblock accepted for publication.

\bibitem{Tz1996}
N.~Tzanakis.
\newblock Solving elliptic {D}iophantine equations by estimating linear forms
  in elliptic logarithms. {T}he case of quartic equations.
\newblock {\em Acta Arith.}, 75(2):165--190, 1996.

\end{thebibliography}
\end{document}